\documentclass[a4paper,11pt]{article}

\usepackage[scale=0.68]{geometry}
\usepackage{amsmath}
\usepackage[parfill]{parskip}    %
\usepackage{graphicx}
\usepackage{amssymb}
\usepackage{epstopdf}
\usepackage{float}
\usepackage{caption}
\usepackage{hyperref}
\usepackage{mathtools}
\usepackage{color}
\usepackage{listings}

\usepackage{libertine}

\makeatletter
\renewcommand{\maketitle}{%
  \noindent
  \hspace*{-1cm}%
  \begin{minipage}{\dimexpr\textwidth + 2cm\relax}%
    \centering
    {\LARGE\bfseries \@title \par}%
    \vskip 1.5em
    {\large \@author \par}%
    \vskip 0.5em
    {\large \@date \par}%
  \end{minipage}%
  \vskip 2em
}
\makeatother

\usepackage[normalem]{ulem}
\usepackage{ mathrsfs }
\usepackage{bbm}
\usepackage[dvipsnames]{xcolor}
\usepackage{todonotes}
\usepackage{units}

\usepackage{amsthm}
\usepackage{thmtools}
\usepackage{thm-restate}

\newtheorem{theorem}{Theorem}[section]
\newtheorem{corollary}[theorem]{Corollary}
\newtheorem{proposition}[theorem]{Proposition}

\newtheorem{lemma}[theorem]{Lemma}

\newtheorem*{theorem*}{Theorem}

\theoremstyle{remark}
\newtheorem{remark}[theorem]{Remark}
\AtBeginEnvironment{remark}{%
  \pushQED{\qed}%
}
\AtEndEnvironment{remark}{\popQED}%

\theoremstyle{definition}
\newtheorem{definition}[theorem]{Definition}
\AtBeginEnvironment{definition}{%
  \pushQED{\qed}%
}
\AtEndEnvironment{definition}{\popQED}%

\usepackage{ wasysym }
\usepackage{ textcomp }

\newcommand{\1}{\mathbbm{1}}

\newcommand{\cH}{\mathcal{H}}

\newcommand{\cO}{\mathcal{O}}

\newcommand{\cL}{\mathcal{L}}

\newcommand{\de}{\operatorname{d}}

\DeclareMathOperator{\e}{\mathbb{E}}
\DeclareMathOperator{\p}{\mathbb{P}}

\usepackage{enumitem} 
\usepackage{multirow}

\usepackage{framed}
\usepackage{float}
\usepackage[caption = false]{subfig}
\usepackage{graphicx}
\usepackage{makecell}

\usepackage{appendix}

\usepackage{tcolorbox}

\usepackage{orcidlink}

\usepackage[alphabetic,nobysame,msc-links,initials,short-months]{amsrefs}
  
\usepackage[affil-sl]{authblk}

\title{The Supermarket Model on a Dynamic Regular Hypergraph}
\author[1]{John Fernley \orcidlink{0000-0002-6635-4341}}
\author[2]{Bal\'azs Gerencs\'er \orcidlink{0000-0002-3885-4146}}
\affil[1]{Centre for Research in Statistical Methodology, Department of Statistics,}
\affil[ ]{University of Warwick,
Coventry,
United Kingdom}
\affil[ ]{\href{mailto:john.fernley@warwick.ac.uk}{\rm john.fernley@warwick.ac.uk}}
\affil[2]{HUN-REN Alfr\'ed R\'enyi Institute of Mathematics,
            Budapest, %
            Hungary}
\affil[ ]{ELTE E\"otv\"os Lor\'and University,
            Budapest, %
            Hungary}
\affil[ ]{\href{mailto:gerencser.balazs@renyi.hu}{\rm gerencser.balazs@renyi.hu}}
\date{\today}

\begin{document}
\maketitle
\begin{abstract}
\noindent
The supermarket model is a system of $n$ queues each with serving rates $1$ and arrival rates $\lambda$ per vertex, where tasks will move on arrival to the shortest adjacent queue. 
We consider the supermarket model in the small $\lambda$ regime on a large dynamic configuration hypergraph with stubs swapping their hyperedge membership at rate $\kappa$.

This interpolates previous investigations of the supermarket model on static graphs of bounded degree (where an exponential tail produces a logarithmic queue) and with independently drawn neighbourhoods (where the ``power of two choices'' phenomenon is a doubly logarithmic queue). We find with high probability, over any polynomial timeframe, the order of the longest queue is
\[
\log\log n + \frac{\log n}{\log \kappa} \wedge \log n
\]
so in the sense of controlling the order of maximal queue length, we identify which speed orders are sufficiently fast that there is no gain in moving the environment faster. Additional results describe mixing of the system and propagation of chaos over time.

\vspace{0.5em}

{\scriptsize

}
\end{abstract}

\refstepcounter{section}
The most efficient way to organise queueing among $n$ rate $1$ Poisson servers is simply to assign tasks to join the shortest queue (JSQ). Then, for any $\lambda \in (0,1)$ when tasks arrive at Poisson rate $\lambda n$ we find mean waiting time $0$ as $n\rightarrow\infty$. Further, we have waiting time stochastically no longer for JSQ than any other queue assignment system.

Realistically however, in some large systems, this protocol of joining the shortest queue might require a prohibitively large communication overhead that moreover needs to happen very fast to not delay service.
For this reason, the supermarket model has been considered by \cite{mukherjee2018asymptotically} on a static graph with Erd\H{o}s-R\'enyi distribution of mean degree $\mu$. They find the same limiting SDE for the rescaled proportions of queue lengths as on a complete graph (i.e. for JSQ), but for this they must take $\mu =\omega(\sqrt{n}\log n)$.

\begin{definition}[Supermarket model]
The supermarket model on a (dynamic hyper-)graph puts a Poisson server of rate $1$ at every vertex, and also at every vertex a Poisson process of task arrivals at rate $\lambda>0$. These determine the queue lengths $Q_t^{(v)}$ at each vertex $v \in [n]$ and time $t\geq 0$ with the following interaction: at the arrival of each task, the task joins a uniform queue among the shortest in the neighbourhood.
\end{definition}

Alternatively to that polynomially large degree, a lesser communication overhead can be achieved by selecting $d\in \mathbb{N}$ other uniform random queues to compare to on every arrival and joining the shortest of those, a.k.a. JSQ$(d+1)$. This was seen in \cite{vvedenskaya96} to immediately bring the exponential tail of the stationary queue length for $d=0$ (the system of independent queues) down to a doubly exponential tail.

This doubly exponential tail corresponds over polynomial timeframe to a longest queue of length $\Theta(\log \log n)$, as was controlled very precisely in \cite{mcdiarmid2006}. However in terms of system architecture we are now required to move every incident connection at every arrival, and still this must happen before the task can be assigned. Some compromise can be achieved with a dynamic graph that moves edges independently between arrivals, and in this article we see that it is not necessary to go all the way to $\kappa=\infty$ (or JSQ$(d+1)$) to control the longest queue at length $O(\log \log n)$.

The supermarket model has been recently considered in a dynamic environment in \cite{goldsztajn2023load}, however their model must resample the whole graph at once from some fixed graph law. This is a different way to slow down the $\kappa=\infty$ supermarket model, necessitating of course a factor $n$ more edges changed to break a neighbourhood of large queues. Regardless, their interest is in a different set of results, primarily in the fluid limit, and the longest queue is not addressed.

To be specific, rather than resampling the full graph, we put our supermarket model on the following random hypergraph, where each arrival to a vertex is joining the shortest of $2r$ available queues. When $r=1$, this is just a uniform random matching.

\begin{definition}[Random equal partition]\label{def_partition}
For $r \in \mathbb{N}$ take vertex set $[n]=\{1,\dots,n\}$ with $n$ divisible by $2r$. Uniformly split $[n]$ into disjoint sets of size $2r$, which are thought of as hyperedges.
\end{definition}

This is the stationary distribution of the following graph dynamic.

\begin{definition}[Dynamic equal partition]\label{def_dynamic_partition}
From any state that could be generated in Definition \ref{def_partition}, each vertex elects to move at rate $\kappa$. It then selects a uniform vertex in $[n]$ and swaps its partition membership with that vertex. In this way the hypergraph changes at exactly rate $\kappa (n - 2r)$.
\end{definition}

What we have arrived at is a natural hypergraph dynamic, and more a natural modification to the supermarket model of \cite{mcdiarmid2006,mcdiarmid_chaos}: their model samples independently a comparison set for each arrival, but here we resample the vertices one-by-one to look at the comparative effect of a slower moving hypergraph.

\section{Main results}\label{sec_results}

Our main result is on the longest queue observed. We state it in the general Landau notation, with the addition of a subscript as in the two-sided order $\Theta_{\p}$ if the implied inequalities only hold with high probability.

\begin{theorem}\label{thm_queuelength}
Take $C > 0$, $r\geq 1$ and $\lambda>0$ sufficiently small. 
On the dynamic equal partition of \ref{def_dynamic_partition}, the supermarket model $(Q^{(v)}_t)_{v \in [n]}$ with arrival rate $\lambda$ has

\[
\max_{t\in [0,n^C]}
\max_{v\in [n]}
Q^{(v)}_t=
\begin{cases}
\Theta_{\p}\left(\log \log n\right), &\kappa=n^{\Omega\left(\nicefrac{1}{\log \log n}\right)},\\[5pt]%
\Theta_{\p}\left(\frac{\log n}{\log \kappa}\right), &\kappa\in\left(\omega(1),n^{o\left(\nicefrac{1}{\log \log n}\right)}\right),\\[5pt]
\Theta_{\p}\left(\log n\right), &\kappa=O(1),%
\end{cases}
\]
for $\kappa=\kappa(n)\in[0,\infty)$.
\end{theorem}

This can be compared to the following result for the infinite speed version of this model. To use our notation but remove the requirement of even sized hyperedges, we will allow half-integers $r$.

\begin{theorem}[{\cite[Theorem 1.3]{mcdiarmid2006}}]\label{thm_mcdiarmid_length}
Take $r\in \tfrac{1}{2}\mathbb{N}^+$, $C > 0$ and $\lambda\in(0,1)$. 
The dynamic environment makes $2r-1$ independent selections of a uniform neighbour on every arriving task. Then we find the supermarket model $(Q^{(v)}_t)_{v \in \mathbb{N}}$ with arrival rate $\lambda$ has 
\[
\max_{t\in [0,n^C]}
\max_{v\in [n]}
Q^{(v)}_t=
\log_{2r} \log n+O_{\p}(1).
\]
\end{theorem}

It can be seen that ``taking $\kappa=\infty$'' has allowed more precise understanding of the localisation, but note of course that this result connects to the large $\kappa$ order $\Theta_{\p}\left(\log \log n\right)$ in Theorem \ref{thm_queuelength}.

We write $\pi=\pi_\kappa^{(n)}$ for the stationary distribution of a queue, by exchangeability the same for any queue, and the pointwise limit
$
\eta_\kappa=\lim_{n\rightarrow\infty}\pi.
$ 

By dominated convergence we have that $\pi\stackrel{L^1}{\rightarrow}\eta_\kappa$ as $n\rightarrow\infty$. In the case $r=1$, we also find concentration of the empirical measure in Corollary \ref{cor_concentration}---from this a law of large numbers would follow immediately. %

For the infinite speed queueing system, which is really a complete graph model, there is the following result on long-term closeness to independence. This result is described as \emph{propagation of chaos} \cite{MR1108185} because the opposite of chaoticity is dependence.

\begin{theorem}[{\cite[Theorem 1.4]{mcdiarmid_chaos}}]
In the same context as Theorem \ref{thm_mcdiarmid_length}, we have at stationarity 
\[
\de_{\rm TV}
\left(
\cL (Q^{(1)} , \dots , Q^{(x)})
,\eta_\infty^{\otimes x}
\right)
=O\left(
\frac{1}{n}\log^2 n (2\log \log n)^{x+1}
\right)
\]
uniformly over $x \in [n]$.
\end{theorem}

This problem, in spite of the presence of a dynamic graph, is much simplified by our small $\lambda$ assumption and so we can state the following additional result.

\begin{proposition}\label{prop_indep}
In the same context as Theorem \ref{thm_queuelength}, we have at stationarity
\[
\de_{\rm TV}
\left(
\cL (Q^{(1)} , \dots , Q^{(x)})
,\eta_\kappa^{\otimes x}
\right)
=O\left(
\frac{x^2}{n}
\right)
\]
uniformly over $x \in [n]$.
\end{proposition}

This follows from considering when the graph union of $2r$ copies of the construction of Section \ref{sec_dependence} finds a tree. Finally, the construction of Section \ref{sec_dependence} tells us that the queue process typically mixes in time $\Theta(\log n)$. The mixing time of an infinite space process does require some definition, however, as from very large initial queues we can only serve at rate $1$ so the \emph{worst case initial position} mixing time would be infinite.

\begin{definition}[Mixing time]
The cartesian product of the queueing system and dynamic hypergraph from some initial graph $g$ and initial queue $q$ has law at time $t$ 
\[
\cL\Big(
\big(Q^{(i)}_t\big)_{i \in [n]},
G_t
\Big|q,g\Big)
\]
and consequently mixing time defined by the total variation distance $\de_{\rm TV}$%
\[
t_{\rm mix}(q,g)=\inf\left\{
t \geq 0 :
\de_{\rm TV}\left[
\cL\Big(
\big(Q^{(i)}_t\big)_{i \in [n]},
G_t
\Big|q,g\Big),
\Pi
\right]\leq \frac{1}{e}
\right\}.
\]
\end{definition}

Moreover, while the mixing time of an unlabelled dynamic configuration model is not well understood \cite{cooper07,erdos22}, within our restricted space of labelled half-edges we have mixing in $\Theta(n\log n)$ moves. This is because our dynamic is essentially several copies of a random transposition Markov chain. Hence we can state the following mixing result.

\begin{proposition}\label{prop_mixing}
In the same contexts as Theorem \ref{thm_queuelength}, from queue lengths $q$ with $\max q =O(\log n)$ and from arbitrary equal partition $g$
\[
t_{\rm mix}(q,g)=\Theta\left(\Big(1+\frac{1}{\kappa}\,\Big)\log n\right).
\]
\end{proposition}

This too extends the $\Theta(\log n)$ which we see in \cite[Theorem 1.1]{mcdiarmid2006}, for what we have heuristically identified as the $\kappa=\infty$, $r=1$ case of this model.

\subsection{General degree}

Sticking together multiple copies of the model of Definition \ref{def_dynamic_partition} we can construct $d$-regular $2r$-uniform dynamic hypergraphs. In this more general graph, the supermarket model then continues to distribute tasks among the neighbourhood of the vertex to which they arrive.

\begin{definition}[Random $d$-regular $2r$-partite hypergraph]\label{def_hypergraph}
For $r, d\in \mathbb{N}$ take vertex set $[n]=\{1,\dots,n\}$ with $n$ divisible by $2r$. Each hyperedge is a subset of $2r$ vertices (the hypergraph is $2r$--uniform) such that every vertex is in $d$ hyperedges (the hypergraph is $d$--regular). Make this construction uniformly at random and we have a random simple regular uniform configuration hypergraph.
\end{definition}

Equivalently, this model could be defined as the configuration model hypergraph conditioned on being simple (with the generalisation of simplicity to hypergraphs in the sense of \cite{frieze13}). However, without this conditioning there is a subtle difference in the frequency of repeated pairs of vertices. Note that this is perhaps the less common generalisation of simplicity, for example \cite{voloshin09,bretto13} use the term to disallow inclusion of edges.

The dynamic with the above stationary distribution is as follows.

\begin{definition}[Dynamic $d$-regular $2r$-partite hypergraph]\label{def_dynamic_hypergraph}
Label hyperedges incident to each vertex with $[d]$. Then, at each vertex, each ``half-edge'' of label $i \in [d]$ at rate $\kappa$ picks a uniform vertex in $[n]$ and swaps with the half-edge of label $i$ incident to that vertex.
\end{definition}

Note this construction is possible (i.e. we can label the hyperedges in this way) for graphs in the range of Definition \ref{def_hypergraph} as they are a union of $d$ copies of the graph of Definition \ref{def_partition}. Note also that the half-edge moves of its own volition as often as by being picked, so the total rate for each edge is $2\kappa(1-\nicefrac{1}{n})$.

\begin{remark}
Constructively, each label corresponds to one of $d$ $1$-regular configuration models and the dynamic describes independent swapping of these $d$ bijections. Combining sets of $r$ edges in each configuration model then turns the graph into a $2r$-regular hypergraph which never multiply features the same vertex in a hyperedge.
\end{remark}

\begin{proposition}\label{prop_generalisation}
Take now general $d\geq 1$:
    \begin{enumerate}[label=(\alph*), ref=\ref{prop_generalisation}(\alph*)]
        \item Theorem \ref{thm_queuelength} holds for the dynamic $d$-regular $2r$-partite hypergraph; \label{prop_general_queuelength}
        \item Proposition \ref{prop_indep} holds for the dynamic $d$-regular $2r$-partite hypergraph; \label{prop_general_indep}
        \item Proposition \ref{prop_mixing} holds for the dynamic $d$-regular $2r$-partite hypergraph. \label{prop_general_mixing}
    \end{enumerate}
\end{proposition}

These statements include the previous results as their $d=1$ case, so we will only prove them in this form.

\section{The dependence structure}\label{sec_dependence}

The assumption of small $\lambda$ is very powerful for controlling dependence in the system, a construction for which we will discuss here. 
Additionally, note that for this model we have $1+d(2r-1)$ vertices in the neighbourhood: we will alleviate notation in this section and the next by writing \[m:=1+d(2r-1).\]

Fortunately, in the later upper bounds of Section \ref{sec_upper} we have a stochastic inclusion argument to be able to set $d=r=1$ and so forget $m$. 

\begin{lemma}
When $m\lambda<1$, and two random variables on $\mathbb{N}^n$ are ordered $q\preceq\tilde{q}$, we have domination of the supermarket model $(Q^{(v)}_t)_{v\in[n],t\geq 0}$ from $Q_0=q$ by a system $(\tilde{Q}^{(v)}_t)_{v\in[n],t\geq 0}$ from $\tilde{Q}_0=\tilde{q}$. The dynamics of $\tilde{Q}$ are that of $n$ independent biased random walks, where each decreases at rate $1$ and increases at rate $m\lambda$.
\end{lemma}

\begin{proof}
$\tilde{Q}$ is simply a collection of M/M/1 queues, which we construct by putting at each vertex in $[n]$ a Poisson process of rate $m\lambda$ of arriving tasks and a Poisson process of rate $1$ which contains the serving times. Then $Q$ can be constructed from $\tilde{Q}$:
\begin{itemize}
\item label each task in $\tilde{Q}$ arriving to some $v\in[n]$ with a uniform variable in $[m]$, corresponding to a position in the neighbourhood of $v$;
\item this task is accepted, i.e. it is allowed to increase $Q^{(v)}$, if $v$ is the unique shortest queue length in the neighbourhood of that $m$\textsuperscript{th} vertex;
\begin{itemize}
\item[$\hookrightarrow$] if there are $k$ jointly shortest queues including $v$, it is instead accepted independently with probability $\nicefrac{1}{k}$.
\end{itemize}
\end{itemize}

With a strict subset of the arrival times and the exact same serving times, the domination is preserved.
\end{proof}

Then, think of the system which is stationary at time $0$ as extended in both directions to a stationary process on times in $\mathbb{R}$. We can do the same for a stationary version of $\tilde{Q}$ while preserving the domination. 
$\tilde{Q}$ has sets of past zero hitting times
\[
\cO^{(i)}:=\left\{
t\leq 0:
\tilde{Q}_t^{(i)}=0, \,
\tilde{Q}_{t-}^{(i)}=1
\right\}
\]
where the sets $\{\cO^{(i)}:i \in [n]\}=:\cO$ are also independent random variables. Explore from some vertex $j \in [n]$ backwards in time in the graphical construction:
\begin{itemize}
\item initially vertex $j$ is \emph{active};
\item if a queuer arrives in the neighbourhood of an active vertex then all vertices in the neighbourhood of arrival become active;
\item If vertex $i$ is active at a time in $\cO^{(i)}$ it  becomes \emph{dormant} at that time (it may become active again).
\end{itemize}

We write \emph{neighbourhood} above to mean the set of a vertex and its neighbours. 
When there are no active vertices remaining, the set of dormant vertices is denoted $\cH^{(j)}$---see a sample path in Figure \ref{fig_dependence}. This is similar to the information sets of \cite{keliger2024concentration}, improved for our specific model with the independence given by the sets $\cO^{(i)}$.

\begin{figure}
\centering
\includegraphics[width=0.5\textwidth]{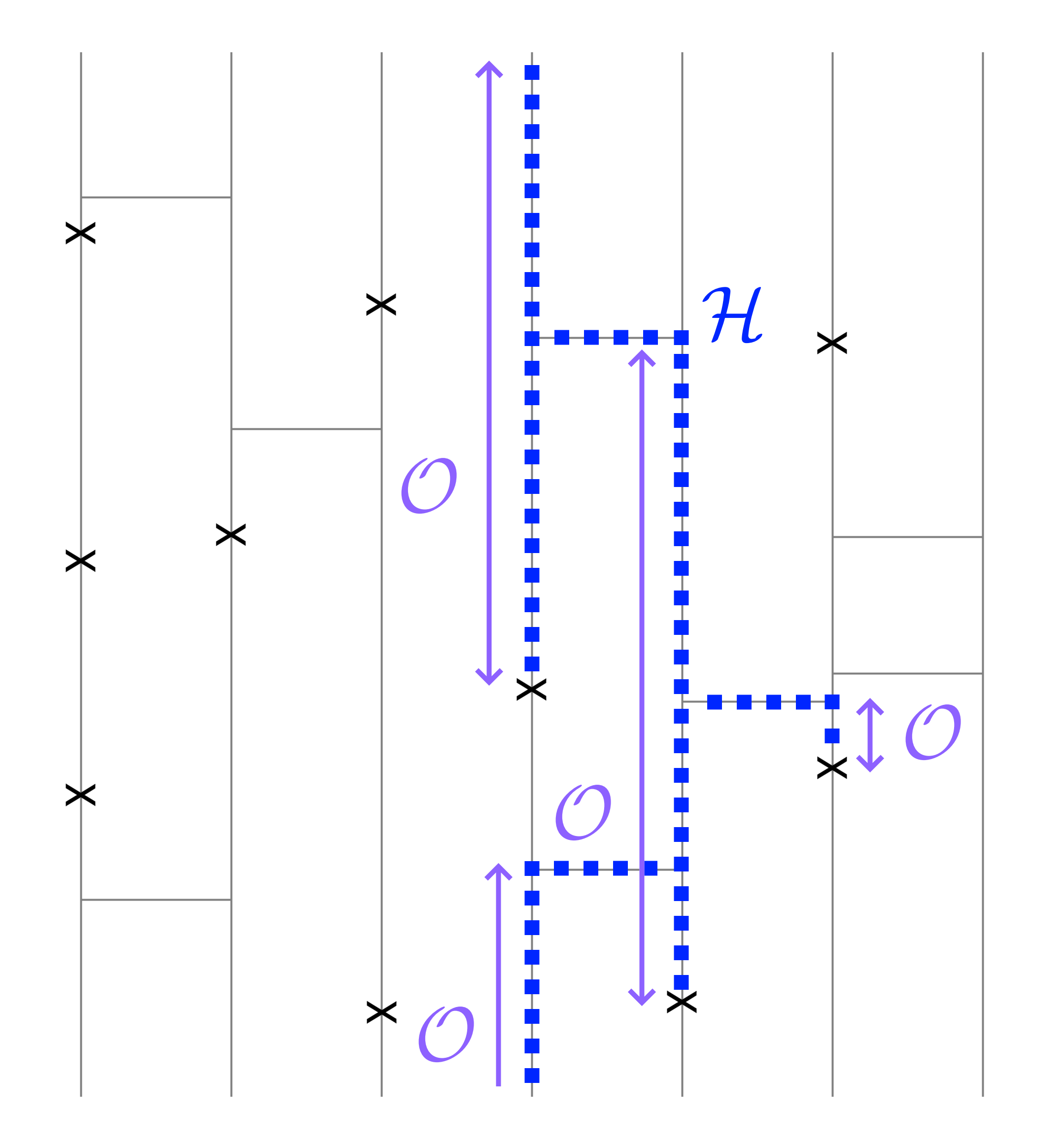}
\caption{The construction of Section \ref{sec_dependence}, with horizontal lines marking shared arrivals and crosses marking times in $\cO$. Then on top, the dotted line is the exploration of $\cH$ finding $3$ vertices, and the arrows measure the lengths of $-t-\cO^{(i)}_{-t}$. The sampled offspring are then $1$, $2$, $0$ and $0$.}\label{fig_dependence}
\end{figure}

Thus we see every hitting in $\cO$ that had an effect on the queue at time $0$. 
The length of the queue of interest is completely determined by the times $\cO$ at the boundary of the time-space component we explored and the Poisson arrivals inside it of the graphical construction.  
So, two queues with totally distinct labels in their backwards explorations are functions of independent random variables, hence on this event they are conditionally independent (conditioning on the full graph history).

%The following expression is also seen in \cite{garden}*{Section 10}.
%
%\begin{lemma}\label{lem_gw_moment}
%A Galton-Watson process $(Z_n)_{n \geq 0}$ with ${Z_0=1}$, $\e(Z_1)=\mu<1$ and $\var(Z_1)=\sigma^2$ has
%\[
%\var
%\left(
%\sum_{n=1}^\infty Z_n
%\right)
%=\frac{\sigma^2}{(1-\mu)^3}.
%\] 
%\end{lemma}
%
%\begin{proof}
%From the calculation of \cite{kortchemski}*{Proof of Lemma 2.10}
%\[
%\begin{split}
%\e\left(\left(
%\sum_{n=1}^\infty Z_n
%\right)^2\right)&=
%\left(
%1+\sum_{n=1}^\infty
%\left(
%\frac{\sigma^2 \mu^{n-1}(1-\mu^n)}{1-\mu}+\mu^{2n}
%\right)
%\right)
%\left(
%1+\frac{2\mu}{1-\mu}
%\right)\\
%&=
%\left(
%1+
%\frac{\sigma^2}{(1-\mu)^2(1+\mu)}
%+\frac{\mu^2}{1-\mu^2}
%\right)
%\left(
%1+\frac{2\mu}{1-\mu}
%\right)\\
%&=\frac{1}{(1-\mu)^2}
%+\frac{\sigma^2}{(1-\mu)^3}
%\end{split}
%\]
%where we identify the first term as the mean size of the tree, squared.
%\end{proof}

This next claim is where this construction requires $\lambda$ sufficiently small. The exponential tails of the M/M/1 queues lead to exponential tails of the dependence component size. This will be used to control mixing, and to control independence in space.

\begin{proposition}\label{prop_max_cpt}
For any $i \in [n]$ and $k \geq 2$
\[
\p\left(|\cH^{(i)}| \geq k \right) \leq  (4m+1)(18m^2\lambda)^{k-1}
\]
in the stationary realisation and for $\lambda$ small enough.
\end{proposition}

%\begin{lemma}\label{lem_h_squared}
%For any $i \in [n]$ and at stationarity,
%$\e(|\cH^{(i)}|^2)\leq 2$ when $\lambda$ is sufficiently small (depending on $r$ and $d$).
%\end{lemma}

\begin{proof}
%For most of this proof, we take $r=d=1$. 
From some queue length $\tilde{Q}_{-t}^{(i)}=x$, the previous hitting time \[\cO^{(i)}_{-t}:=\max\cO^{(i)}\cap(-\infty,-t)\] is given exactly by a continuous-time biased random walk $Y$. In discrete time, consider how many steps the jump chain must take to hit $0$. From $1$ we consider a $1$-step recursion
\[
\e_1(s^{T_0})
=
\frac{s}{1+m\lambda}
+
\frac{m\lambda s}{1+m\lambda}
\e_1(s^{T_0})^2.
\]

This moment can only be finite for $s$ small enough, i.e. when the quadratic has a non-negative discriminant
\[
s\leq\frac{1+m\lambda}{2\sqrt{m\lambda}}
\]
and so we can take $s=\nicefrac{1}{(2\sqrt{m\lambda})}$. Then solve the quadratic to find:
\[
\e_1(s^{T_0})
=
\frac{1}{\sqrt{\lambda m}}
+
\sqrt{\lambda m}
+
\sqrt{2+\lambda m}
\leq
\frac{1}{\sqrt{\lambda}}
\]
because $\lambda$ is sufficiently small and $m\geq 2$. 
Hence from general $x \geq 1$, the hitting time is a sum of $x$ independent samples of the distribution above i.e.
$
\e_x(s^{T_0})
\leq
\left(
\nicefrac{1}{\sqrt{\lambda}}
\right)^{x}.
$ 

To explore $\cH$ we must explore backwards once until every discovered vertex hits a time in $\cO^{(i)}$, and then forwards from that time at every discovered vertex at most up to $t=0$. To simplify this, we can upper bound the structure by exploring forwards and backwards at every discovered vertex.

So, we are interested in how many arrivals a particular vertex $i$ had in between the previous time in $\cO^{(i)}$ and the next time in $\cO^{(i)}$. The biased random walk is the same backwards in time, and so the length of that interval (given $x$) has the generating function of two independent copies: $(\nicefrac{1}{\lambda})^x$. For the arrivals $D_x$ this means
\[
\e_x\left((2\sqrt{m\lambda})^{-2x-2D_x}
\right)
=
\e_x\left((4m\lambda)^{-x-D_x}
\right)
\leq
\left(
\frac{1}{\lambda}
\right)^x
\]
because the initial queue length $x$ must be served in both directions, and any arrivals take unit time and then additional unit time to be served.

In a stationary history, the first discovery of vertex $i$ has at most queue length drawn from $\tilde{Q}^{(i)}$ at stationarity, and subsequent rediscoveries were forgotten at a hitting of $0$ so their new state can also be dominated by stationarity. However the act of discovering them may add $1$ to the queue length, so we start at $1$ larger than stationarity
\[
\tilde{g}(x)=(1-m\lambda)(m\lambda)^{x-1}\1_{k\geq 1}
\]

Therefore an explored stationary queue has offspring distribution $D$ with
\[
\begin{split}
\e(4m\lambda)^{-D}
&\leq
\e_{\tilde{g}}
(4m\lambda)^{-D_x}
=\sum_{x=1}^\infty
(1-m\lambda)(m\lambda)^{x-1}
\e_x(4m\lambda)^{D_x}\\
&=
\left(\frac{1}{m\lambda}-1\right)
\sum_{x=1}^\infty
(4m^2\lambda^2)^{x}
\e_x(4m\lambda)^{-x-D_x}\\
&\leq
\left(\frac{1}{m\lambda}-1\right)
\sum_{x=1}^\infty
(4m^2\lambda)^{x}
\leq 4m+1.
\\
\end{split}
\]

We explore this Galton--Watson tree by its \L ukasiewicz path: starting at $1$ representing the root, iteratively explore a vertex and increment the path by the number of children you find there (new vertices to be explored) minus $1$ for the vertex just explored. 
The component is completely explored when the \L ukasiewicz path hits $0$, and at $k$ explorations it is positive with probability at most
\[
\begin{split}
\p\left(\sum_{i=1}^k X_i \geq k-1 \right)
=\p\left((4m\lambda)^{-\sum_{i=1}^k X_i} \geq (4m\lambda)^{1-k} \right)
&\leq(4m\lambda)^{k-1}(4m+1)^k\\
\end{split}
\]
by Markov's inequality, which is the claimed bound because $4m(4m+1)\leq 18m^2$.
\end{proof}

%\begin{proof}[excess proof]
%and so by Markov's inequality
%\begin{multline*}
%\p_{\tilde{g}}(k\text{ or more up-steps in the interval of }\cO^{(i)})
%\leq
%\e_{\tilde{g}}\left(
%s^{-x-2k}\left(\frac{4}{\lambda}\right)^x
%\right)\\
%=\sum_{x=1}^\infty
%(1-2\lambda)(2\lambda)^{x-1}
%s^{-x-2k}
%\left(
%\frac{4}{\lambda}
%\right)^{x}\\
%=
%\left(\frac{1}{2\lambda}-1\right) (3\sqrt{\lambda})^{2k}
%\sum_{x=1}^\infty
%\left(
%2\lambda \cdot 3\sqrt{\lambda} \cdot \frac{4}{\lambda}
%\right)^x
%=(9\lambda)^{k}
%\left(
%\frac{12}{\sqrt{\lambda}}+O(1)
%\right)
%.
%\end{multline*}
%
%\hrulefill
%
%What that means is that we accept every arrival in the interval as in the previous display: the exploration of $\cH$ is contained in a Galton--Watson tree with that offspring distribution.
%
%An $\mathbb{N}$-valued variable with $\p(D\geq k)\leq (9\lambda)^{k}
%\left(
%\nicefrac{13}{\sqrt{\lambda}}
%\right)$ necessarily has %
%\[
%\e(D)\leq\e(D^2)\leq
%\sum_{k=1}^\infty
%2k
%(9\lambda)^{k}
%\frac{13}{\sqrt{\lambda}}
%=
%O(\sqrt{\lambda}).
%\]
%
%
%So far this was with $r=d=1$ but general $r$ and $d$ just introduce constant factors in the moments above that do not affect the orders. In the end by Lemma \ref{lem_gw_moment} the second moment of the number of vertices in the full tree is $1+O(\sqrt{\lambda})$.
%\end{proof}

We immediately derive the following bound.

\begin{corollary}\label{cor_h_squared}
For any $i \in [n]$ and at stationarity,
$\e(|\cH^{(i)}|^2)\leq 2$ when $\lambda$ is sufficiently small (depending on $r$ and $d$).
\end{corollary}

\begin{proof}
Insert Proposition \ref{prop_max_cpt} into the expression
$
\e(\cH^2)=1+
\sum_{k=2}^\infty
(2k-1)
\p(\cH \geq k).
$
\end{proof}

Now we can prove the result from Section \ref{sec_results} on independence of sampled queues, by checking when their information sets are disjoint.

\begin{proof}[Proof of Proposition \ref{prop_general_indep}]
At stationarity, the queueing system can be generated with the exploration of $\cH=\{\cH^{(i)}:i \in [n]\}$, finding new conditionally independent queues at every new vertex.

The same can be done for the $n=\infty$ version of this model where every vertex has a new label. In that version by construction we never repeat a label and every root had independent queue length drawn from $\eta_\kappa$, but if the finite $n$ version of the model never repeats a label then we can couple the two explorations.

From Corollary \ref{cor_h_squared}, this expected number of label clashes is at most
\[
\frac{1}{n}
\e
\left(
\sum_{i=1}^x
\cH^{(i)}
\right)^2
=
O\left(
\frac{x^2}{n}
\right).
\]

We conclude by Markov's inequality that $x$ subcritical explorations from distinct vertices repeat a label with the claimed probability. This proves Proposition \ref{prop_indep}, and the argument is unaffected by the addition of $d-1$ extra layers to the model.
\end{proof}

The previous proof used the information sets $\cH$ to separate over the space of the graph, but if we think of this separation in just the time direction then what we find is a bound on mixing.

\begin{proof}[Proof of Proposition \ref{prop_general_mixing}]
In each $2r$-hyperedge put an arbitrary perfect matching of $r$ edges. In this way the dynamic of Definition \ref{def_dynamic_hypergraph} can be coupled to $r \times d$ dynamic perfect matchings.

Think of each matching as encoded by a length $n$ list of all its edges, so that an odd index is paired with the consecutive even index. Then the dynamic on this list is exactly a uniform random transposition (allowing probability $\nicefrac{1}{n}$ that an element swaps with itself and so there is no transposition). Hence by the union bound and \cite{diaconis1981generating}, from arbitrary initial graph $g$, $r \times d$ such shuffles will mix in time $O(\nicefrac{\log n}{\kappa})$. Moreover one of these hyper-edges will not move until time $\Omega_{\p}(\nicefrac{\log n}{\kappa})$ and so the system cannot have mixed.

After graph mixing, allowing further time $C\log n$ with $C$ large enough, we are with probability $1-\epsilon$ in an event where exploring $\cH$ doesn't go back in time far enough to see the graph before stationarity. This follows from the exponential tails of $\cO^{(i)}$ intervals combined with Proposition \ref{prop_max_cpt} to say the number of relevant $i$ also has an exponential tail. Hence, the current state is independent from $q$.

For the lower bound in the case of small $C$, when also $\kappa\rightarrow\infty$ so that the lower bound on this term is meaningful, we see a queue that has had no arrivals or servings and so still has its initial value from $q$.
\end{proof}

\section{Lower bounds on queue length}

As we remarked in Section \ref{sec_results}, we will prove only the lower bound for Proposition \ref{prop_general_queuelength} in this section, that is the lower bounds for the graph model of Definition \ref{def_dynamic_hypergraph}. We continue in this section to write $m:=1+d(2r-1)$.

\begin{definition}[zero-on-update process]
This Markov chain on $\mathbb{N}^{m}$ is a system of queues on the \emph{$2r$-uniform hyperstar}, i.e. the graph of a $1$-ball in our system in which a central vertex has $m-1$ unique neighbours distributed equally among $d$ hyperedges. The first dimension of $\mathbb{N}^{m}$ is the central queue.

Arrivals and servings on this graph are as usual with the supermarket model (rates $\lambda$ and $1$ per vertex, with moving to the shortest neighbour), but we simplify updates. An update at rate $\kappa$ of each neighbour zeroes that incident queue, or at the centre we have $d$ clocks (each also of rate $\kappa$) that each zero all $2r-1$ incident queues of their edge.
\end{definition}

On the full system, we could restrict arrivals to only rate $\lambda$ at the central vertex to be able to lower bound the system by the independent product of $n$ of the processes above. In this section there is no need, we can demonstrate the lower bound on a single queue.

\begin{proposition}\label{prop_zeroqueue}
The zero-on-update process has marginal stationary distribution $\rho$ of the central queue, with 
\[
\rho(x)\geq
\frac{\lambda\left(
1-\lambda m
\right)}{m(\lambda m+1)^2}
\left(
\frac{\lambda m}{\lambda m+m+2rd\kappa}
\right)^{x m}.
\]
\end{proposition}

\begin{proof}
Tasks arrive at most at rate $\lambda m$, and are completed at rate $1$. 
Hence for $\lambda m<1$ we get by comparison to the biased random walk
\begin{equation}\label{eq_rw_with_m}
\rho(0)\geq 1-\lambda m,
\end{equation}
from where the first queue increases to length $1$ at rate at least $\lambda$ and then leaves $1$ at rate at most $\lambda m+1$. Therefore
\[
\rho(1)\geq \frac{\lambda}{\lambda m+1}\left(
1-\lambda m
\right).
\]

From central queue length $k\geq 1$, which we can lower bound by state $(k,0,\dots, 0)$ on the $m$ neighbourhood queues, leaving the state follows a path to increase the first queue to $k+x$ in an event of probability
\[
\frac{1}{m}
\left(
\frac{\lambda m}{\lambda m+m+2rd\kappa}
\right)^{x+(m-1)(k+x-1)}
\]
after which we are in state $(k+x,k+x-1,\dots,k+x-1)$. The first queue then leaves state $k+x$ at rate bounded by $\lambda m+1$ and so this event shows
\[
\rho(k+x) \geq
\frac{1}{m(\lambda m+1)}
\left(
\frac{\lambda m}{\lambda m+m+2rd\kappa}
\right)^{x+(m-1)(k+x-1)}
\rho(k)
\]
because we also leave any state $k\geq 1$ at rate at least $1$. Setting $k=1$ gives the claimed inequality.
\end{proof}

Already we can prove the lower bound orders in our main theorem.

\begin{corollary}\label{cor_lower_bound}
From initially the empty queue system, over any timescale $n^{\Theta(1)}$ any particular queue sees length
\[
\Omega_{\p}\left(
\frac{\log n}{\log \kappa}
\right)
\]
except that it sees $\Omega_{\p}\left(
\log \log n
\right)$ when
\[
\log \kappa = \Omega\left(
\frac{\log n}{\log \log n}
\right)
\]
and $\Omega_{\p}(\log n)$ when $\kappa=O(1)$.
\end{corollary}

\begin{proof}
\textbf{1.} First consider $\kappa=O(1)$. Split the graph into $\Omega(n)$ disjoint neighbourhood, and lower bound the sum of a neighbourhood by the biased random walk $W_t$ decreasing at rate $m$ and increasing at rate $\lambda$, up until the neighbourhood is broken at rate $\kappa r d$. 

This walk has no downsteps and we also see no update, in time $t=\epsilon\log n$, with probability
\[
e^{-(m+\kappa r d)t}
=
n^{-\epsilon(m+\kappa r d)}
\]
so for small enough $\epsilon$ we have $n^{\Omega_{\p}(1)}$ such queues. Each lower bound walk conditionally has Poisson distribution with mean $t\lambda$ and so with high probability one of them is $\Omega_{\p}(\log n)$.

\textbf{2.} Next consider $\kappa$ at least a constant but still $
\log \kappa = O\left(
\frac{\log n}{\log \log n}
\right).
$ Proposition \ref{prop_zeroqueue} gives the result here, because the zero-on-update process lower bounds the length of a particular queue.

More precisely, it bounds the stationary mass and so the expected total length of time that a particular queue spends at the claimed order. This is turned into a high probability lower bound by using the mixing of Proposition \ref{prop_general_mixing} to resample periodically from that stationary mass, where for any $\epsilon>0$ there is some level $\Omega\left(
\nicefrac{\log n}{\log \kappa}
\right)$ such that we find each sample has at least $n^{-\epsilon}$ chance to hit that level.

\textbf{3.} For the remaining case of larger order $\kappa$, we consider the event that, at a queue of length $k$, before an arrival at the queue we see $m-1$ updates and find $m-1$ neighbours of the same length. During this process, repeated updates at the $2rd$ half-edges that make up the neighbourhood are fine: we check once that each neighbour has length $k$ at the time of the final arrival. What we must not see is a service at any of the $m$ vertices while each updates, and then one arrival at the centre. This produces
\[
\begin{split}
\pi(k+1)
&\geq
\pi(k)^{m}
\left(
\frac{\kappa}{\kappa+m}
\right)^{m-1}
\frac{\lambda}{\lambda+m}%
\geq
\frac{\pi(k)^{m}}{\left(
1+\frac{m}{\kappa}
\right)^{m-1}}
\cdot
\frac{\lambda}{m+\lambda}
\end{split}
\]
which we can iterate to find
\[
\begin{split}
\pi(k)
&\geq
\pi(0)^{m^k}
\left[
\frac{1}{\left(
1+\frac{m}{\kappa}
\right)^{m-1}}
\cdot
\frac{\lambda}{m+\lambda}
\right]^{1+m+\dots+m^{k-1}}\\
&=
\left(
\frac{\pi(0)}{1+\frac{m}{\kappa}}
\right)^{m^k}
\left(
1+\frac{m}{\kappa}
\right)
\left[
\frac{\lambda}{m+\lambda}
\right]^{\frac{m^k-1}{m-1}}\\
&\geq
\left(
\frac{\pi(0)}{1+\frac{m}{\kappa}}
\cdot
\frac{\lambda^{1/m}}{m^{1/m}}
\right)^{m^k}
\left(
1+\frac{m}{\kappa}
\right)
\left[
\frac{\lambda}{m+\lambda}
\right]^{\frac{-1}{m-1}}.%
\end{split}
\]

Recall from \eqref{eq_rw_with_m} that $\pi(0)\geq 1-m\lambda>0$, requiring of course that $\lambda$ is sufficiently small. So we infer from the above display, for any small $\epsilon>0$, we can find some level $k=\Omega_{\p}\left(
\log \log n
\right)$ with $\pi(k)\geq n^{-\epsilon}$. We conclude 
 using the mixing of Proposition \ref{prop_general_mixing}, as in the previous case.%
\end{proof}

\section{Upper bounds on queue length}\label{sec_upper}

In this section we prove the upper bounds, which are more difficult. 
First we have an observation of the stochastic comparison that is available here, an extension of what \cite[Theorem 4]{turner98} states for the $\kappa=\infty$ model.

\begin{proposition}\label{prop_domination}
The number of customers with position at least $k$ in the supermarket model on the dynamic simple regular uniform configuration hypergraph (Definition \ref{def_dynamic_hypergraph} for any $r,d\geq 1$) is stochastically upper bounded by the number of customers with position at least $k$ in the supermarket model on a dynamic perfect matching (Definition \ref{def_dynamic_partition} with $r=1$) whenever the initial conditions are the same.

Hence, by considering the large time limit, we have domination of the two stationary distributions. Therefore we have domination of the two processes from their respective stationarity distributions.
\end{proposition}

\begin{proof}
Because our hyperedges have even order $2r$, we can always cut them up uniformly into $r$ edges in an associated configuration model. Do this just for the hyperedges of type $1$, with the same dynamic of swapping half-edges. As the half-edges move together, the subgraph relationship is preserved. Couple the two dynamics in this way, and also couple the serving times and arrivals.

Then, because the subgraph relationship is preserved, any arrival of a task in the two models will always have a selection of queues in the hypergraph that includes the two queues available in the graph, and so will join a shorter or equal queue in the hypergraph than in the graph.
\end{proof}

Of course, this domination means that the maximum, i.e. the $k$ for which number of customers with position at least $k+1$ is $0$, is also in stochastic comparison between the two environments. Therefore to upper bound the maximum, in this section, we can take $r=d=1$.

We can now proceed with the analysis of a supermarket model on the dynamic $1$--regular configuration model. Consider the stationary distribution of a particular queue $\pi$, by exchangeability the same for every queue. Define also the distribution
\[
\eta_\kappa=\lim_{n\rightarrow\infty}\pi
\]
by the pointwise limit. This is also the stationary distribution of the $n$-intertwined mean field approximation for this model \cite{keliger2024concentration}. %

\begin{remark}\label{rem_random_walk}
When $\lambda<\nicefrac{1}{2}$, we can stochastically dominate the queue at a vertex by supposing that it all possible tasks (at rate $2\lambda$) join that queue. Hence we have stochastic dominations $\pi\preceq g$ and $\eta_\kappa\preceq g$ by the biased random walk stationary distribution
\[
g(x)=(1-2\lambda)(2\lambda)^{x}\1_{x\geq 0}.
\]

In this section we discuss only the dynamic perfect matching, but by Proposition \ref{prop_domination} this upper bound also applies to the model of Definition \ref{def_dynamic_hypergraph} which contains the matching.
\end{remark}

 For the simplified $n=\infty$ context we can say more in the following proposition.

\begin{proposition}\label{prop_stat_upper_bound}
For $\lambda$ small enough, $\kappa$ large enough and some $R=\Theta\left(\log \log \kappa \right)$ we have
\[
\eta_\kappa([k,\infty))\leq 
\begin{cases}
2\left(
2\lambda
\right)^{2^k-1}, & k\leq R,\\\left(
\nicefrac{3\lambda}{\kappa}
\right)^{ \nicefrac{k}{2} }, & k > R.\\
\end{cases}
\]
\end{proposition}

We do not claim the bound for every $R=\Theta\left(\log \log \kappa \right)$, just that there is some changepoint $R$ defined by \eqref{eq_R_bound} for which there is this bound.

\begin{proof}
\textbf{1.} We use induction. 
Recall in \cite[Equation 3]{turner98} we see the stationary queue distribution when $\kappa=\infty$ has expression
\[
\eta_\infty([k,\infty))=
\lambda^{2^k-1}
\]
which is the solution of equating rates between $k-1$ and $k$
\begin{equation}\label{eq_meanfield}
\lambda\left(
\eta_\infty([k-1,\infty))^2
-
\eta_\infty([k,\infty))^2
\right)
=
\eta_\infty(k).
\end{equation}

In the finite speed version which is our model, call a vertex ${\tt new}$ if it has recently changed neighbour, such that it has not seen an arrival to interact with its neighbour. Otherwise, the vertex is ${\tt old}$. We then split the stationary distribution
\begin{equation}
\label{eq_splitting}
\eta_\kappa(k)=\sigma(k,{\tt old})
+
\sigma(k,{\tt new})
\end{equation}
for every $k\geq 0$. We will use this for a similar argument to \eqref{eq_meanfield} of equating rates at stationarity.

Because any stationary new vertex is adjacent to a stationary other new vertex, we upper bound the ergodic rate of increasing from $k-1$ to $k$ in our model by the left side of the following inequality
\begin{multline*}
2\lambda\sigma(k-1,{\tt old})
+
\lambda
\sigma(k-1,{\tt new})
\left(
\sigma(k-1,{\tt new})
+
2\sigma([k,\infty),{\tt new})
\right)\\
\geq
\sigma(k,{\tt old})
+
\sigma(k,{\tt new})
\geq
\sigma(k,{\tt new}).
\end{multline*}

Regardless of the queue length we have a simple Markov transition between ${\tt new}$ and ${\tt old}$
\[
(1+2\lambda)
\sigma(\mathbb{N},{\tt new})
=
2\kappa
\sigma(\mathbb{N},{\tt old})
\implies
\sigma(k-1,{\tt old})
\leq
\frac{1+2\lambda}{2\kappa}
\]
so we can insert this
\[
\frac{\lambda}{\kappa}(1+2\lambda)
+
\lambda
\sigma(k-1,{\tt new})
\left(
\sigma(k-1,{\tt new})
+
2\sigma([k,\infty),{\tt new})
\right)
\geq
\sigma(k,{\tt new})
\]
which implies
\[
2\lambda
\sigma(k-1,{\tt new})
\left(
\sigma(k-1,{\tt new})
+
2\sigma([k,\infty),{\tt new})
\right)
\geq
\sigma(k,{\tt new}),
\]
the same relation as \eqref{eq_meanfield} with a doubled task arrival parameter, if
\[
\sigma(k-1,{\tt new})^2
\geq
\frac{1+2\lambda}{2\kappa}.
\]

In this way we can stop the induction below that level, at some $R=\Theta(\log\log \kappa)$ with
\begin{equation}
\label{eq_R_bound}
\sigma(R,{\tt new})\leq \frac{1}{\sqrt{2\kappa}}.
\end{equation}

At every step the transition upwards is more than \eqref{eq_meanfield}, so the conclusion is a stochastic domination as claimed. 
%Beyond this point $R$, we consider the crossings of an interval $[k,k+2]$. Servings are at rate $1$, and updates have no effect on the length of a particular queue, so we multiply that rate $1$ by the probability to serve again without an arrival: downward crossings occur at rate \emph{at least}
%\[
%\sigma(k+2,{\tt new})
%\frac{1}{2\lambda+1}.
%\]
An old state becomes new by updating, and so
\begin{equation}\label{eq_refreshing}
(1+2\lambda)
\sigma(k,{\tt new})
\geq
2\kappa
\sigma(k,{\tt old})
\end{equation}
this is not detailed balance (the chain is not reversible), instead the left-hand side is the total rate of leaving $(k,{\tt new})$. By recalling \eqref{eq_splitting} we can translate the bound back to $\eta_\kappa$ and lose at most the claimed $2$ factor.

\textbf{2.} For the rest of the proof, think of this as just a Markov process with stationary distribution $\tilde{\sigma}(k)\propto \sigma(k,{\tt new})$: skip forwards through time intervals of ${\tt old}$ to create the \emph{partially observed} chain \cite[Definition 3.3]{fernley2022discursive}.

We argue first, by switching to the random walk upper bound as soon as the queue length hits $k$, to find, for $\lambda<\nicefrac{1}{2}$, stochastic domination $\eta_\kappa \preceq \rho$ with
\[
\rho(i)=
\begin{cases}
\eta_\kappa(i), &i \leq k,\\
\eta_\kappa(k)(1-2\lambda)(2\lambda)^{i-k}, &i > k.\\
\end{cases}
\]

Hence $\eta_\kappa([k,\infty))\leq (1+2\lambda) \eta_\kappa(k)$ independently of the inductive argument to follow. Moreover by \eqref{eq_refreshing}
\[
%\tilde{\sigma}([k,\infty)) \leq
%\frac{1+2\lambda+2\kappa}{2\kappa}
\eta_\kappa([k,\infty))
\leq
 (1+2\lambda) \tilde{\sigma}(k) \left(
 1+\frac{1+2\lambda}{2\kappa}
 \right).
\]

We now consider the rates from $k-i$ into $[k,\infty)$. For this process a queue increases by $i\geq 1$ with a neighbour of length $k-i-j$ (for $j\geq 0$) using at least $2i-1+j$ arrivals. This transition occurs then at rate upper bounded by
\[
(2+2\lambda) \cdot
\eta_\kappa(k-i-j)
\cdot
\left(
\frac{\lambda}{\kappa}
\right)^{2i-1+j}.
\]

Moreover, we can increase by sampling a strictly larger queue, at rate loosely bounded by
\[
\eta_\kappa([k-i+1,\infty))
\left(
2\lambda \cdot
\left(
\frac{\lambda}{\kappa}
\right)^{i-1}
+
2 \cdot
\left(
\frac{\lambda}{\kappa}
\right)^{i}
\right)
\]
because starting a path with a serving necessitates at least $1$ extra arriving task. 
In the other direction, we leave the set $[k,\infty)$ with rate at least
$
\tilde{\sigma}(k)
$. Therefore we must have
%\begin{multline*}
%\tilde{\sigma}(k)
%\leq
%\sum_{i= 1}^k
%\tilde{\sigma}(k-i)
%\Bigg[
%(1+2\lambda)
%\eta_\kappa(k-i+1)
%\left(
%2\lambda 
%\left(
%\frac{\lambda}{\kappa}
%\right)^{i-1}
%+
%2 
%\left(
%\frac{\lambda}{\kappa}
%\right)^{i}
%\right)\\
%+
%\sum_{j\geq 0}
%(2+2\lambda) 
%\eta_\kappa(k-i-j)
%\left(
%\frac{\lambda}{\kappa}
%\right)^{2i-1+j}
%\Bigg]
%\end{multline*}
\begin{multline*}
\tilde{\sigma}(k)
\leq
\sum_{i= 1}^k
\tilde{\sigma}(k-i)
\Bigg[
\eta_\kappa([k-i+1,\infty))
(2+2\lambda) 
\left(
\frac{\lambda}{\kappa}
\right)^{i-1}\\
+
\sum_{j\geq 0}
(2+2\lambda) 
\eta_\kappa(k-i-j)
\left(
\frac{\lambda}{\kappa}
\right)^{2i-1+j}
\Bigg].
\end{multline*}

Sum over $j$, and loosely bound $\eta_\kappa\leq 1$ to simplify, except in the extracted $\tilde{\sigma}(k)$ term on the right-hand side
\begin{multline*}
\tilde{\sigma}(k)
\leq
(2+2\lambda) 
\tilde{\sigma}(k-1)
(1+2\lambda) \tilde{\sigma}(k) \left(
 1+\frac{1+2\lambda}{2\kappa}
 \right)\\
+
\sum_{i= 2}^k
(2+2\lambda) 
\tilde{\sigma}(k-i)
\left(
\frac{\lambda}{\kappa}
\right)^{i-1}
+
\sum_{i= 1}^k
(2+2\lambda) 
\tilde{\sigma}(k-i)
\frac{(\nicefrac{\lambda}{\kappa})^{2i-1}}{1-\nicefrac{\lambda}{\kappa}}.
\end{multline*}

Recall now \eqref{eq_R_bound}
\[
\tilde{\sigma}(R)= \sigma(R,{\tt new})\frac{1+2\lambda+2\kappa}{2\kappa}
%\leq \frac{1+2\lambda+2\kappa}{(2\kappa)^{\nicefrac{3}{2}}}
=O\left(\frac{1}{\sqrt{\kappa}}\right)
\]
and so as $\lambda\rightarrow 0$, $\kappa\rightarrow\infty$ we have uniformly over $k\geq R+1$
\[
\begin{split}
\left(
1-O\left(\frac{1}{\sqrt{\kappa}}\right)
\right)
\tilde{\sigma}(k)
&\leq
\sum_{i= 2}^k
(2+2\lambda) 
\tilde{\sigma}(k-i)
\left(
\frac{\lambda}{\kappa}
\right)^{i-1}
+
\sum_{i= 1}^k
(2+2\lambda) 
\tilde{\sigma}(k-i)
\frac{(\nicefrac{\lambda}{\kappa})^{2i-1}}{1-\nicefrac{\lambda}{\kappa}}\\
&\leq
\left(
2+O\left(\lambda\right)
\right)
\left(
\tilde{\sigma}(k-1)
\frac{\lambda}{\kappa}
+
\sum_{i= 2}^k
\tilde{\sigma}(k-i)
\left(
\frac{\lambda}{\kappa}
\right)^{i-1}
\right)
\lesssim
\frac{2\lambda}{\kappa}
\tilde{\sigma}(k-2).
\\
\end{split}
\]

This inductively proves the claimed decay, the bound applies to the ${\tt old}$ states as well by \eqref{eq_refreshing}, and to the stationary mass of interest which is the sum of ${\tt new}$ and ${\tt old}$.
\end{proof}

\begin{lemma}\label{lem_tv}
In the stationary queue system
\[
\de_{\rm TV}(\pi, \eta_\kappa)\leq
\frac{1}{n}
.
\]
\end{lemma}

\begin{proof}
Explore the finite $n$ graphical construction of Section \ref{sec_dependence} to locally construct the $n=\infty$ version. As long as every sampled label after a graph move is new to the construction, we can consistently do this. 
Note also the distributions $\pi$ and $\eta_\kappa$ are for a single queue, which by exchangeability could be any $i \in [n]$.

Labelling $\cH^{(i)}$ iteratively, we expect at most $\nicefrac{|\cH^{(i)}|^2}{2n}$ clashes and so by Markov's inequality and Corollary \ref{cor_h_squared} we have none with probability $1-\nicefrac{1}{n}$.
\end{proof}

The previous result showed that the stationary distribution of a queue in the system is approximated by its mean-field version. In the next, we instead approximate the empirical distribution of all $n$ queues sampled together.

\begin{lemma}\label{lem_moment_empirical}
In the stationary queue system for any $a\geq 1$
\[
\mathbb{E}\left[\left(
\sum_{x=0}^\infty \left|
\frac{1}{n}\sum_{i=1}^n \1_{Q^{(i)}=x}
-\eta_\kappa(x)
\right|\right)^{a}\right]
\leq
\frac{\log^{3a} n}{n^{\nicefrac{a}{8}}}
\]
when $n$ is taken sufficiently large (depending on $a$).
\end{lemma}

\begin{proof}
Fix some natural number $k$, later this will become $k=\lceil\nicefrac{a}{8}\rceil$. 
We split for some sufficiently large $C>0$
\[
\begin{split}
&\mathbb{E}\left[\left(
\sum_{x=0}^{\infty} \left|
\frac{1}{n}\sum_{i=1}^n \1_{Q^{(i)}=x}
-\eta_\kappa(x)
\right|
\right)^{8k}\right]\\
&\leq
2^{4k}\mathbb{E}\left[\left(
\sum_{x=0}^{\lfloor C \log n \rfloor} \left|
\frac{1}{n}\sum_{i=1}^n \1_{Q^{(i)}=x}
-\eta_\kappa(x)
\right|
\right)^{8k}
+
\left(
\sum_{x=\lceil C \log n \rceil}^{\infty} \left|
\frac{1}{n}\sum_{i=1}^n \1_{Q^{(i)}=x}
-\eta_\kappa(x)
\right|
\right)^{8k}\right]
\end{split}
\]
because the binomial expansion has $2^{8k}$ terms. 
For the tail we know of course
\[
\sum_{x=\lceil C \log n \rceil}^{\infty}
\left(
\frac{1}{n}\sum_{i=1}^n \1_{Q^{(i)}=x}
+\eta_\kappa(x)
\right)\leq 2
\]
and hence using Remark \ref{rem_random_walk}
\[
\begin{split}
\mathbb{E}\left[\left(
\sum_{x=\lceil C \log n \rceil}^{\infty} \left|
\frac{1}{n}\sum_{i=1}^n \1_{Q^{(i)}=x}
-\eta_\kappa(x)
\right|
\right)^{8k}\right]
&\leq
2^{8k-1}\mathbb{E}\left(
\sum_{x=\lceil C \log n \rceil}^{\infty}
\frac{1}{n}\sum_{i=1}^n \1_{Q^{(i)}=x}
+\eta_\kappa(x)
\right)\\
&=
2^{8k-1}
\left(
\pi\left(
[C \log n,\infty)
\right)
+
\eta_\kappa\left(
[C \log n,\infty)
\right)
\right)\\
&\leq 
2^{8k}
(2\lambda)^{C \log n}
=
2^{8k}
n^{-C \log \frac{1}{2\lambda}}
.
\end{split}
\]

For the rest, we bound
\begin{equation}\label{eq_one_term}
\left(
\sum_{x=0}^{\lfloor C \log n \rfloor} \left|
\frac{1}{n}\sum_{i=1}^n \1_{Q^{(i)}=x}
-\eta_\kappa(x)
\right|
\right)^{8k}
\leq
C^{8k} \log^{8k} n 
\max_{x \leq C \log n}
 \left(
\frac{1}{n}\sum_{i=1}^n \1_{Q^{(i)}=x}-\eta_\kappa(x)
\right)^{8k}
\end{equation}
and so for polynomial decay we only need to control the moment of a particular queue length $x\leq C \log n$. 

Now if we write $X_i:=\1_{Q^{(i)}=x}-\eta_\kappa(x)$, expanding $(\sum_{i=1}^n X_i)^{8k}$ has $n^{8k}$ terms. Most of these terms, given $k$ is constant, have all distinct factors. We control the proportion that have at most $4r$ factors with exponent $1$, for some $r < 2k$ to be optimised later. Such a term with $s$ distinct factors must have
\[
4r+2(s-4r) \leq 8k
\iff
s \leq 4k+2r
\]
and so there are at most
\[
\binom{n}{4k+2r}(4k+2r)^{8k}=
O\left(
n^{4k+2r}
\right)
\]
such terms. Bound the expectation of the modulus of these terms by $1$.

For the other terms, we explore $\cH$ and check if these factors with exponent $1$ are \emph{separated} in that they find no other labels in their $\cH$ component corresponding to any of the the $8k-1$ other factors. Specifically, we control that at least $2r$ of the $4r$ factors with exponent $1$ are separated.

This must follow if there at most $r$ label clashes, which fails with probability bounded by
\begin{multline*}
\p\left(
\max_i \cH^{(i)} \geq z
\right)
+
\p\left(
{\rm Pois}
\left(
\frac{(8kz)^2}{n-8kz}
\right)
\geq r+1
\right)\\
\lesssim
8k(72\lambda)^{z-1}
+
\frac{1}{(r+1)!}
\left(
\frac{(8kz)^2}{n-8kz}
\right)^{r+1}
=O\left(
n^{-2r}
\right)
\end{multline*}
by using Proposition \ref{prop_max_cpt} and setting $z=2r \log n$. 
Similarly for any $i \in [n]$ we have
\[
\p(\{i\text{ separated}\}^\complement)
\lesssim
8k(72\lambda)^{z-1}
+
\frac{8kz^2}{n-8kz}
=O\left(
\frac{\log^2 n}{n}
\right).
\]

Now conditionally on the event that they are separated, the queues determining these $r$ singleton factors are independent. Further
\[
\begin{split}
\pi(x)-\eta_\kappa(x)&=\p(\{i\text{ separated}\})\e(X_i|\{i\text{ separated}\})
 +\p(\{i\text{ separated}\}^\complement)\e(X_i|\{i\text{ separated}\}^\complement)
\end{split}
\]
which together with $|X_i|\leq 1$ and Lemma \ref{lem_tv} guarantees
\[
\left|
\e(X_i|\{i\text{ separated}\})
\right|
\leq
\frac{\p(\{i\text{ separated}\}^\complement)+\frac{1}{n}}{\p(\{i\text{ separated}\})}=O\left(\frac{\log^2 n}{n}\right).
\]

In this way $2r$ separated factors have expectation $O\left(\frac{\log^{4r} n}{n^{2r}}\right)$: we find below
\[
\begin{split}
\e\left[
 \left(
\frac{1}{n}\sum_{i=1}^n \1_{Q^{(i)}=x}-\eta_\kappa(x)
\right)^{8k}
\right]
&\leq
\underbrace{
\frac{1}{n^{8k}}
O\left(
n^{4k+2r}
\right)
+O\left(\frac{\log^{4r} n}{n^{2r}}\right)
}_{\text{the two types of term}}
+
\underbrace{
O\left(\frac{1}{n^{2r}}\right)
}_{\text{the error event}}\\
&=O\left(\frac{1}{n^{4k-2r}} + \frac{\log^{4r} n}{n^{2r}}\right).
\end{split}
\]

So, we set $r=k$ to obtain the bound
\[
\mathbb{E}\left[\left(
\sum_{x=0}^\infty \left|
\frac{1}{n}\sum_{i=1}^n \1_{Q^{(i)}=x}
-\eta_\kappa(x)
\right|\right)^{8k}\right]
\leq\frac{\log^{12k+1} n}{n^k},
\]
recalling the $\log^{8k} n$ factor due to \eqref{eq_one_term}.

Finally, because the $L^p$ norms are ordered we can take $1\leq a \leq 8k$ and deduce
\[
\mathbb{E}\left[\left(
\sum_{x=0}^\infty \left|
\frac{1}{n}\sum_{i=1}^n \1_{Q^{(i)}=x}
-\eta_\kappa(x)
\right|\right)^{a}\right]
\leq
\left(
\frac{\log^{12k+1} n}{n^k}
\right)^{\frac{a}{8k}}
\leq
\frac{\log^{3a} n}{n^{\nicefrac{a}{8}}}.
\]
\end{proof}

This bound on the moment of the distance can be translated to a bound on the tail, to be able to understand a sample drawn from the empirical distribution. 

\begin{corollary}\label{cor_concentration}
For the stationary system and any $x>0$
\[
\p\left(\de_{\rm TV}
\left(
\frac{1}{n}\sum_{i=1}^n \delta_{Q^{(i)}}
,\eta_\kappa
\right)>n^{\nicefrac{-1}{9}}
\right)=O\left(
n^{-x}
\right).
\]
\end{corollary}

\begin{proof}
We insert Lemma \ref{lem_moment_empirical} into Markov's inequality to find
\[
\begin{split}
\p\left(\de_{\rm TV}
\left(
\frac{1}{n}\sum_{i=1}^n \delta_{Q^{(i)}}
,\eta_\kappa
\right)>n^{\nicefrac{-1}{9}}
\right)
&\leq
\frac{
\e\left(
\de_{\rm TV}
\left(
\frac{1}{n}\sum_{i=1}^n \delta_{Q^{(i)}}
,\eta_\kappa
\right)^{a}
\right)
}{n^{\nicefrac{-a}{9}}}\\
&\leq
\frac{\log^{3a} n}{n^{\nicefrac{a}{72}}}
\end{split}
\]
and so we have the conclusion by taking $a$ sufficiently large.
\end{proof}

\begin{remark}
Large deviations for the empirical distribution can be trivialised by the assumption of small $\lambda$ and the random walk upper bound in Remark \ref{rem_random_walk}, but only beyond a minimal distance. That is to say, by considering the $0$ length queues and Hoeffding's inequality, for some binomial $B\sim{\rm Bin}(n,1-2\lambda)$
\begin{multline*}
\p\left(
\sum_{x=0}^{\infty} \left|
\frac{1}{n}\sum_{i=1}^n \1_{Q^{(i)}=x}
-\eta_\kappa(x)
\right|
>4\lambda+\epsilon
\right)\\
\leq
\p\left(
\left|
\frac{B}{n}-1+2\lambda
\right|
+
\left|
\frac{B}{n}-1
\right|
-2\lambda>\epsilon
\right)
\leq
3e^{-n\epsilon^2/2.
}
\end{multline*}

A real large deviation result for deviations smaller than $4\lambda$ might be achieved with a reworking of the previous lemma such as to take the sum for an exponential moment, but the details are complicated and it is not our focus here.
\end{remark}

We are now ready to prove the upper bound of the main theorem.

\begin{proof}[Proof of Proposition \ref{prop_general_queuelength}]
\textbf{1.} Proposition \ref{prop_general_queuelength} inherits a timescale $[0,n^C]$ from the statement of Theorem \ref{thm_queuelength} over which we must control the longest queue. Recall that the lower bound was demonstrated at this particular queue in Corollary \ref{cor_lower_bound}, and that the upper bound need only be proved for the case $r=d=1$ by Proposition \ref{prop_domination}.

For the upper bound, we first eliminate the extremes of $\kappa$:
\begin{enumerate}
\item with general $\kappa$ we can apply Remark \ref{rem_random_walk} to the stationary distribution to find by Markov's inequality that the queue length does not exceed $O_{\p}(\log n)$, which proves the theorem in the case $\kappa=O(1)$;
\item for large enough $A>0$ and $\kappa \geq n^{A}$, recall from the proof of Proposition \ref{prop_general_mixing} (or really from \cite{diaconis1981generating}) that the system mixes in time $O(n^{-A}\log n)$. Hence by \cite[Proposition 1.10(b)]{MR1071805} we can find a strong stationary time before time $n^{1-A}$ with superpolynomially decaying failure probability. Simultaneously the graph sees an arriving task in that period with probability asymptotic to $\lambda n^{2-A}$. 
 In this way, for $A$ large enough, the graph is found independently stationary between every arrival of a task, with high probability over the whole history $[0,n^C]$. So increasing the speed further has no effect at all on the queues in that high probability event.  \label{item_2} 
\end{enumerate}

\textbf{2.} For the remaining intermediate speeds we show that the maximum length of queue $1$ does not exceed the claimed length orders with probability $1-o(\nicefrac{1}{n})$, to then conclude by the union bound for the full system. These orders are:
\begin{equation}\label{eq_orders}
\begin{cases}
\Theta_{\p}\left(\log \log n\right), &\kappa\in\left(n^{\Omega\left(\nicefrac{1}{\log \log n}\right)},n^{A}\right),\\[5pt]
\Theta_{\p}\left(\frac{\log n}{\log \kappa}\right), &\kappa\in\left(\omega(1),n^{o\left(\nicefrac{1}{\log \log n}\right)}\right).
\end{cases}
\end{equation}

By setting $x$ sufficiently large in Corollary \ref{cor_concentration} we can treat each new neighbour at a vertex as an independent sample from $\eta_\kappa$, except that we must allow an error independently with probability $n^{\nicefrac{-1}{9}}$. All errors are reset by $\cO$ which waits a time with an exponential tail, as we saw in the proof of Proposition \ref{prop_max_cpt}. Hence the longest reset time is bounded with probability $1-o(\nicefrac{1}{n})$ by $A\log n$, for some large $A>0$, and up to that reset time in $\cO$ we sample at most at Poisson rate $4\kappa$ of neighbours that might be longer queues than if drawn from $\eta_\kappa$.

Each such error in sampling a neighbour can only produce as many errors in arrivals as there are arrivals before the neighbour is resampled: that is, ${\rm Geom}\left(\frac{4\kappa(1-\nicefrac{1}{n})}{2\lambda+4\kappa(1-\nicefrac{1}{n})}\right)\preceq{\rm Geom}\left(1-\frac{\lambda}{\kappa}\right)$ errors in arrivals. Most of these geometrics are $0$, so in fact we can thin the neighbours arriving at rate $4\kappa$ to \emph{relevant} neighbours arriving at rate $4\lambda$. This contains all the neighbours that see at least $1$ arrival while connected.

Thinning the sequence of $1+{\rm Pois}(4\lambda A \log n)$ relevant neighbours leaves at most $1+{\rm Pois}(4\lambda A n^{\nicefrac{-1}{9}} \log n)$ with erroneous length. 
We have for some $N\in\mathbb{N}$
\[
\p\left(
1+{\rm Pois}(4\lambda A n^{\nicefrac{-1}{9}} \log n)
\geq N
\right)
\leq
\p\left(
{\rm Pois}(n^{\nicefrac{-1}{10}})
\geq N-1
\right)
=O\left(
n^{\frac{1-N}{10}}
\right).
\]

So for some large constant $N$ we will see a bad neighbour no more than $N$ times between two times in $\cO_1$.

By the Chernoff bound, for some other large constant $A'>0$, a sum of $N$ variables $G_i\stackrel{\rm i.i.d.}{\sim}{\rm Geom}(\nicefrac{\kappa}{(\lambda+\kappa}))$ has
\[
\begin{split}
\p\left(
\sum_{i=1}^{N}
G_i \geq \frac{A' \log n}{\log \nicefrac{\kappa}{\lambda}}
\right)
&=
\p\left(
\left(
\nicefrac{\kappa}{\lambda}
\right)^{\sum_{i=1}^{N}
G_i}
\geq
e^{A' \log n}
\right)\\
&\leq n^{-A'}
\left(
\frac{\kappa}{\lambda}
\right)^{2N}
=n^{-\Omega(1)},
\end{split}
\]
for $A'$ sufficiently large (recall that if $\kappa$ is larger than some fixed polynomial exponent then we are already mixing the whole graph between every arrival, by item \ref{item_2} at the start of this proof). Note also that because these are all relevant neighbours, the geometric above has minimum value $1$. So, by the previous display, we can say that the number of errors never exceeds the order of \eqref{eq_orders}.

\textbf{3.} If we have controlled that the errors are negligible, it remains to argue the error-free queue lengths. 
When the drawn queue lengths of the neighbour are simply from $\eta_\kappa$, the length of queue $1$ will also have this distribution. We can apply Proposition \ref{prop_stat_upper_bound} to bound the stationary mass of $\eta_\kappa$ above the claimed orders \eqref{eq_orders} by an arbitrarily fast decaying polynomial. 

Then the total expected time over all queues and the polynomial timeframe $[0,n^C]$ is also polynomially small, and because the queue can only decrease by serving at exponential rate $1$ we have the same polynomially decaying bound on the hitting probability. %
\end{proof}

{\bf Acknowledgements.}

This research was supported by National Research, Development and Innovation Office grant KKP 137490, and by a CRiSM visitor grant.

\bibliography{queuebib}

\end{document}